\documentclass{article}

\usepackage{graphicx}        
\usepackage{multicol}        

\usepackage{amsmath,amsfonts,amsthm,tikz,hyperref,xparse,xspace,amsthm,algpseudocode,mathrsfs,bbm,bbold,dsfont,subcaption}

\newtheorem{theorem}{Theorem}

\newtheorem{proposition}{Proposition}

\usetikzlibrary{cd}
\usepackage{booktabs,pgfplotstable,array}
\pgfplotsset{compat=1.16}

\def\to{\rightarrow}

\def\iff{\Leftrightarrow}

\def\n#1{|| #1 ||}

\def\Id{\mathds{1}}

\def\wPhi{{\underline \Phi}}
\def\wV{{\underline V}}

\def\we{{\underline e}}
\def\wf{{\underline f}}
\def\wg{{\underline g}}
\def\we{{\underline e}}
\def\wr{{\underline r}}
\def\wt{{\underline t}}
\def\wu{{\underline u}}
\def\wv{{\underline v}}
\def\ww{{\underline w}}
\def\wx{{\underline x}}
\def\wy{{\underline y}}
\def\wz{{\underline z}}
\def\wA{{\underline A}}
\def\wB{{\underline B}}

\def\wId{{\underline \Id}}
\def\wT{{\underline T}}

\def\spectrum{\sigma}

\DeclareMathOperator*{\argmin}{argmin}

\DeclareMathOperator{\SPAN}{span}
\def\range{{\mathscr{R}}}

\def\Bs{B_{\text{\sc ssc}}}

\def\wBs{\wB_{\text{\sc ssc}}}

\def\wTs{\wT_{\text{\sc ssc}}}

\let\oequation\equation
\makeatletter
\def\equation{\@ifnextchar[\equation@lab\oequation}
\def\equation@lab[#1]{\oequation \label{eq:#1}}
\makeatother

\begin{document}

\title{Overlapping  subspaces and singular systems  with
  application to Isogeometric Analysis}
\author{Andrea Bressan \\CNR-IMATI ``E. Magenes’’, \\Via Ferrata 5, 27100 Pavia, Italy,\\[.2cm]
   Massimiliano Martinelli\\ CNR-IMATI ``E. Magenes’’, \\ Via Ferrata 5, 27100 Pavia, Italy, \\[.2cm]
Giancarlo Sangalli \\ Department of Mathematics, \\University of
Pavia, Via Ferrata, 5, 27100 Pavia, Italy}

%
\maketitle

\abstract{We propose a framework for solving partial differential
  equations (PDEs) motivated by isogeometric analysis (IGA) and local
  tensor-product splines. Instead of using a global basis for the
  solution space we use as generators the disjoint
  union of subspace bases. This leads to a potentially singular linear
  system, which is handled by a Krylov linear solver. The
  framework may offer computational advantages in dealing with spaces like Hierarchical B-splines, T-splines, and LR-splines.}

\section{Introduction}

In this paper, we propose a novel framework for numerical methods for
partial differential equations (PDEs), inspired by isogeometric
analysis (IGA) based on  local tensor-product splines that allow local refinement.

Consider the problem: find \(u \in V\) such that
\begin{equation}[system-V]
Au=f,
\end{equation}
where \(A:V\rightarrow V^*\) is a linear operator  and \(f\in V^*\),
with \(V\) being a Hilbert space. Assuming \(V\) is
finite-dimensional, one usually rewrite \eqref{eq:system-V}   in
matrix form by choosing a basis \(\Phi\) for \(V\). Instead, we take a
different approach, further assuming that \(V\) is given as the sum,
in general not direct,  of closed subspaces \(V_i \subset V\), \(i = 1, \ldots, n\),
\begin{equation}\label{eq:V}
V = \sum_{i=1}^n V_i,
\end{equation}
and that we have a basis \(\Phi_i\) for each subspace \(V_i\). We then consider the disjoint union
\begin{equation}\label{eq:generating-system}
\wPhi = \coprod_{i=1}^n \Phi_i,
\end{equation}
which generally does not form a basis for \(V\), as its elements may
be linearly dependent (or even duplicated). We represent the solution
\(u\) as a linear combination of elements of \( \wPhi \) and, using \(
\wPhi \) instead of \( \Phi \), rewrite \eqref{eq:system-V}  as a
linear system, which may be singular. Therefore, such a linear system
cannot be solved by standard direct methods but can be handled by
iterative methods. In particular, 
it is known that Krylov methods perform well with
singular systems, provided that the right-hand side lies in the range
of the system matrix, see
e.g. \cite{Kaasschieter1988,Ipsen1998,Reichel2005}.   At the same time,
preconditioned Krylov methods  are currently the most efficient
solvers when a good preconditioner is available.  Following the domain
decomposition approaches, and in particular the  subspace
correction  paradigm, preconditioners can be constructed from the
approximate solvers of the restrictions of \(A\) to each \(V_i\). Our strategy is to
carefully select the \(V_i\), for example, we could consider \(V_i\)
with a tensorial structure, that can be exploited for (possibly
inexact) fast solvers. 

The framework outlined above is unconventional but not unprecedented
in the literature. Griebel's 1994 paper \cite{griebel1994multilevel}
in the context of finite elements proposed abandoning the basis in
favor of a generating system as in (3) to simplify preconditioner
construction. Subsequent research has utilized generating systems in
wavelet methods and sparse grid contexts.

We think this framework is potentially very beneficial for
isogeometric analysis (IGA). IGA, introduced in 2005
\cite{hughes2005,cottrell2009isogeometric}, aims to bridge the gap
between numerical simulations and Computer-Aided Design
(CAD). Specifically, IGA follows the isoparametric paradigm and adopts
B-splines (or NURBS and other spline extensions) as basis
functions. This allows the unknown fields in numerical simulations to
be represented by splines, just as splines are used to parameterize
geometry in CAD. 
In both IGA and CAD, each patch is typically derived from a tensor-product basis. However, spline extensions for mesh refinement such as
Hierarchical B-splines
\cite{forsey1995surface,buffa2016adaptive},
T-splines \cite{sederberg2003t,MR3084741}, and
LR-splines \cite{MR3019748,MR3146870}, feature a local
tensor-product structure. These spaces may be seen as unions of tensor-product subspaces \(V_i\), but often  identifying a global basis
\(\Phi\) is challenging, while  working with subspace bases \(\Phi_i\)
and their disjoint union \( \wPhi \) is a significant simplification.
Furthermore,  using \(\wPhi\) for representing the unknown, we may
retain the computational advantages related to the  tensor-product
structure of the \(V_i\), e.g., the possibility of fast local solvers.

The structure of this work is the following: Section~2 sets the notation, Sections~3--4 reframe Krylov
methods and subspace correction preconditioners in the proposed
setting, Section~5 contains a numerical benchmark in the framework of
IGA and Section~6 contains our conclusions.


\section{Abstract setting}

We assume  that \(V\) is a finite dimensional vector space
with scalar product $(\cdot,\cdot)_V$ and norm \(\n{\cdot }_V\).
Given the  spaces $V_i$, for $i=1, \ldots,
n$, and $V$, as in \eqref{eq:V}, 
we   introduce the extended space 
\begin{equation}[product-space]
    \wV:=\prod_{i=1}^n V_i,
\end{equation}
endowed with the norm
$
\n{\wv}_\wV:=\left(\sum_{i=1}^n	\n{v_i}_V^2\right)^{1/2}$	, for
all $ 	\wv=(v_1,\dots,v_n) \in \wV
$, and   reformulate the abstract problem \eqref{eq:system-V} on
\(\wV\), which will lead to a linear system for the representation of
the solution with respect to  $\wPhi$ defined in
\eqref{eq:generating-system}, the basis of $ \wV $,  instead of  $\Phi$, the basis of $ V $.     {For this purpose,  each $\wv \in \wV$ is associated to
  a $v \in V$ via the following sum operator  \(S:\wV\to V\) 
\begin{equation}[S-definition]
	S\wv=S(v_1,\dots,v_n)=\sum_{i=1}^n v_i\in V
      \end{equation}   
Its  adjoint  \(S^*:V^*\rightarrow \wV^*=\prod_i V_i^*\) is then
injective and associates to \(f\) the tuple of the restrictions of
\(f\) to the \(V_i\), that is $S^* f=(f|_{V_1},\dots,f|_{V_n})$.
The \emph{extended problem} is the pull-back of {problem} 
\eqref{eq:system-V} on \(\wV\), and reads  
\begin{equation}[system-wV]
\wA \wu=\wf,
\end{equation}
where \(\wf:=S^*f\) and \(\wA:\wV\rightarrow \wV^*\) is defined by 
\begin{equation}[wA]
\wA:= S^* A S
\end{equation} 
that is 
\begin{equation}[wA-bis]
\wA (u_1,\dots,u_n)=(w|_{V_1},\dots,w|_{V_n}), \text{ with } w=\sum_{i=1}^n A u_i.
\end{equation}

The \emph{extended problem} is equivalent to problem  \eqref{eq:system-V}  as stated in the next result.

  \begin{theorem}\label{thm:system-equivalent}
    If $A:V\rightarrow V^*$ is an isomorphism, then
    \begin{equation}[ker-wA]
      \ker \wA = \ker S,
    \end{equation}
    and  
    \begin{equation}[wA-and-A-equivalence]
      Au=f \text{  if and only if } u=S \{\wu\in \wV:\wA\wu=\wf\}.
    \end{equation}
  \end{theorem}
  
\begin{proof}
	Recalling that   \(S\) is surjective and  \(S^*\) is
        injective,  \eqref{eq:ker-wA} follows. Thanks to \(S^*\) being
        injective  and \(A\) an isomorphism, we also have
	\[ \wA\wu = \wf \iff S^* (AS\wu ) = S^* f \iff A S\wu=f \iff u = S\wu,\]
       which is \eqref{eq:wA-and-A-equivalence}.
\end{proof}
\section{Krylov methods}

In this section we revise and adapt to our framework the theory of
preconditioned Krylov methods,
with emphasis on  the Conjugate Gradient (CG) method,
following \cite{gunnel2014note}.
In particular, we  relate Krylov methods for the extended  problem
\eqref{eq:system-wV} to  Krylov method for the original
problem\eqref{eq:system-V}.

Assume that  a left preconditioner   \(\wB:\wV^*\to\wV\) for the extended problem
\eqref{eq:system-wV} is given. We associate to it a left preconditioner  \(B:V^*\rightarrow V\) for the
original problem \eqref{eq:system-V} as follows:
\begin{equation}
  \label{eq:defn-B-from-wB}
  B= S\wB S^*;
\end{equation}
then,  the following  commuting diagram holds
\begin{equation}[operators-cd]
    \begin{tikzcd}[sep=2cm,every label/.append style={font=\normalsize}]
        \wV \arrow[r, "\wA"] \arrow[d, "S"] & \wV^* \arrow[r, "\wB"] & \wV \arrow[d, "S"]
        \\ V\arrow[r, "A"] & V^* \arrow[r, "B"] \arrow[u, "S^*"] & V
    \end{tikzcd}
  \end{equation}
  The original and extended  (left) preconditioned operators are  $T:=BA $  and $ \wT:=\wB\wA
  $, respectively,  such that
  \begin{equation}[TS=SwT] TS=S\wT.
\end{equation}
Finally, the  (left) preconditioned extended
problem reads
\begin{equation}
  \label{eq:prec-ext-prbl}
  \wT \wu = \wB \wf. 
\end{equation}

Being $\wT:\wV\to \wV$  an
endomorphism on $\wV$ (though possibly singular) we can employ Krylov
methods for solving \eqref{eq:prec-ext-prbl}. Denoting $\wy = \wB \wf$,   Krylov methods provide 
at each step \(s\) an approximate solution \(\wu_s\) in the Krylov
subspace
\begin{equation}
  \label{eq:Krylov-space}
  	K_s(\wT,\wy):=\SPAN \{ \wy,\wT\wy,\dots,\wT^{s-1}\wy\},
\end{equation}
such that a suitable optimality condition is satisfied.

The convergence of Krylov methods in the singular case is not
guaranteed  in general:  it may happen  that the solution
$\wx$ of $\wT\wx = \wy $ does not belong to any Krylov space
$	K_s(\wT,\wy)$. Fortunately, this does not happen in our case under
reasonable conditions, as stated in the  following proposition. 
\begin{proposition}\label{prop:krylov-convergence}
  Under the assumption that \(T\) is an isomorphisms, for all \(f
  \in V\)  the system \eqref{eq:prec-ext-prbl} admits a Krylov solution.
\end{proposition}
\begin{proof}
The kernel of a composition always contains the kernel of the first map
and so by definition $$\ker \wT=\ker(\wB S^* A S )\supseteq \ker S.$$
Similarly, since \(T\) is an isomorphism and using \eqref{eq:TS=SwT}
we have 
\[
  \ker S= \ker T S= \ker  S\wT \supseteq \ker\wT.
\]
Therefore
\begin{equation}
  \label{eq:kerTbar}
  \ker \wT=\ker S.
\end{equation}
From \eqref{eq:kerTbar}, then \eqref{eq:TS=SwT}, using that
\(T\) is  an isomorphism, and finally \eqref{eq:kerTbar}  again  we have
\[
  \ker \wT^2=\ker S\wT = \ker TS=\ker S=\ker \wT,
\]
i.e., \(\ker\wT^2=\ker \wT\).
This guarantees that \cite[Theorem 2]{krylov-idea}
 $\wT \wx = \wz$ admits a Krylov solution for any \(\wz \in \range{(\wT)}\).
\end{proof}

The assumption of Theorem \ref{prop:krylov-convergence} is natural in
our framework.
In particular, $A$ is taken non-singular
by assumption,
$\wB$ is often coercive by construction
\begin{displaymath}
  \forall \wg \in \wV^*, \quad \wg \neq 0 \Rightarrow \langle \wB \wg , \wg
  \rangle > 0;
\end{displaymath}
then  $B$ is also coercive (since $ \langle B g , g
  \rangle  = \langle \wB S^* g , S^* g   \rangle  $, and $S^* $ injective) and 
  thus \(B\) is an isomorphism, which makes $T$ an
  isomorphism as well.


  Theorem~\ref{prop:krylov-convergence} 
assumes exact arithmetic and the exact calculation of the matrix and
right hand side of the preconditioned extended problem
\eqref{eq:prec-ext-prbl}. However,  numerical errors typically occur
and, when $\wT$ is singular, they may lead to the unsolvability of the
perturbed system. Assume, for example, that the
left-hand side of  $\wT$  is exactly computed (up to machine
precision) and the right-hand-side is affected by a quadrature
error $ \we$, that is the system reads $\wA \wu = \wf+ \we$. 
Let's further decompose $\we$  as $     \we=
  \we_{\range}+\we_{\mathscr{K}} $, with   $\we_{\range}$   in  the
  range of \(S^*\). The effect of   $\we_{\range}$  can be analyzed by means of standard methods such as the
Strang lemma, while $\we_{\mathscr{K}}$, which  is  outside the range
of \(S^*\) and then of $\wA$,
prevents  the existence of the solution of the perturbed problem. In this situation,
it is required that the stopping criterion of the Krylov iteration is  set with a
threshold on the residual norm that cannot be smaller than the norm of  $\we_{\mathscr{K}}$.

\subsection{CG and MINRES}
If $A$ and $\wB$ are self-adjoint and coercive 
\begin{align}
  \label{eq:sym-pos-A} \langle A v , w  \rangle=\langle A w , v  \rangle,&\qquad \langle A v , v  \rangle >0,\, &&\forall v,w\in V,
  \\\label{eq:sym-pos-wB}\langle \wB \wv , \ww  \rangle=\langle \wB \ww , \wv  \rangle,&\qquad  \langle \wB \wv , \wv  \rangle >0,\, &&\forall \wv,\ww\in \wV^*.
\end{align}
then also $\wA$ is self-adjoint and $B$ is both self-adjoint and coercive 
\begin{align}
  \label{eq:sym-wA} \langle \wA \wv , \ww  \rangle=\langle \wA \ww , \wv  \rangle,& &&\forall \wv,\ww\in \wV,
  \\\label{eq:sym-B}\langle B v , w  \rangle=\langle B w , v  \rangle,&\qquad  \langle B v , v  \rangle >0,\, &&\forall v,w\in V^*.
\end{align}

We can indeed apply CG, for which the convergence histories of the
original and extended problem coincide.
\begin{theorem}\label{thm:krylov-iteration-CG} Under the assumptions above,
  let $u_s$ and $\wu_s$ be the  iterates at step \(s\)  of the
  (preconditioned) CG method applied 
  to the original problem \eqref{eq:system-V} and the extended problem
  \eqref{eq:system-wV}, respectively. Then for all $s$,
  \begin{equation}
    \label{eq:CG-equiv}
    S\wu_s= u_s
  \end{equation}
  \begin{equation}
    \label{eq:CG-equiv-norm}
    \n{\we_s}_{\wA}=\n{e_s}_A, 
  \end{equation}
where \(\we_s\in \wV\)   is  the
error at step \(s\) of the CG method applied to the extended problem,  i.e.,
the difference between a chosen solution \(\wu\) and \(\wu_s\),
while \(e_s\in V\) is  the corresponding error obtained by the CG
method applied to the original problem, and  $\n{\cdot}_{A} =
\langle A \cdot , \cdot  \rangle$ and  $\n{\cdot}_{\wA} =
\langle \wA \cdot , \cdot  \rangle$ denote the norms or seminorms
associated to $A$ and $\wA$,  respectively.
\end{theorem}

\begin{proof}

 For any positive integer $s$, we have by definition of $T$ and $\wT$
 that
 \begin{displaymath}
    T^sS=   S \wT^s, 
  \end{displaymath}
 and then

 \begin{displaymath}
    T^sBf =   S \wT^s\wB \wf, 
  \end{displaymath}
yielding the correspondence of  the Krylov spaces generated in \(V\) 
and \(\wV\):
\begin{equation}
  \label{eq:Krylov-relation}
  K_s(T,Bf)  = S K_s(\wT,\wB \wf).
\end{equation}
By construction any CG solution of \eqref{eq:system-V},
respectively,\eqref{eq:system-wV} fulfill the following optimality condition 
\[  u_s = \argmin_{v \in K_s(T,Bf) } \n{u - v }_{A},\quad\text{and}\quad \wu_s \in \argmin_{\wv \in K_s(\wT,\wB \wf)} \n{\wu - \wv }_{\wA}
\]
By \eqref{eq:Krylov-relation} and
\[
\n{\wv}_{\wA}^2=\langle \wA \wv , \wv  \rangle = \langle A S \wv ,
S \wv  \rangle  = \n{S \wv}_{A}^2.
\]
we have $S\wu_s=u_s$, i.e., \eqref{eq:CG-equiv} and \eqref{eq:CG-equiv-norm}.

\end{proof}

In a similar way, but without the need of coerciviness of
$\langle A \cdot , \cdot  \rangle $, i.e., assuming only that \(A\) is
a self-adjoint isomorphism satisfying
\begin{equation}
  \label{eq:sym-A} \langle A v , w  \rangle=\langle A w , v  \rangle,\qquad \forall v,w\in V,
\end{equation}
we can still apply MINRES and have the following equivalence.

\begin{theorem}\label{thm:krylov-iteration-MINRES} Under the assumptions \eqref{eq:sym-A},
  \eqref{eq:sym-pos-wB},  let $u_s$ and $\wu_s$ be the  iterates at step \(s\)  of the
  (preconditioned) MINRES method applied 
  to the original problem \eqref{eq:system-V} and the extended problem
  \eqref{eq:system-wV}, respectively. Then for all $s$,
  \begin{equation}
    \label{eq:MINRES-equiv}
    S\wu_s= u_s
  \end{equation}
  and 
  \begin{equation}
    \label{eq:MINRES-equiv-norm}
   \n{\wr_s}_{\wB}= \n{r_s}_B,
  \end{equation}
where \(\wr_s\in \wV^*\) is the residual at step \(s\) of the MINRES
method applied to the extended problem, \(r_s\in V^*\)  is the residual of the MINRES method applied to the original problem  and  $\n{\cdot}_{B} =
\langle B \cdot , \cdot  \rangle$ and  $\n{\cdot}_{\wB} =
\langle \wB \cdot , \cdot  \rangle$ denote the norms 
associated to $B$ and $\wB$,  respectively.
\end{theorem}

\begin{proof}
As shown in \eqref{eq:Krylov-relation} the Krylov subspaces correspond by \(S\).
Similarly to the CG, any MINRES solution of \eqref{eq:system-V},
respectively,\eqref{eq:system-wV} fulfill the following optimality conditions
\[
    u_s = \argmin_{v \in K_s(T,Bf) } \n{A v - f }_{B},\quad\text{and}\quad
\wu_s \in \argmin_{\wv \in K_s(\wT,\wB \wf) } \n{\wA \wv - \wf }_{\wB}.
\]
By \eqref{eq:Krylov-relation}, \(\n{S^*v }_{\wB}= \n{v}_{B}\)
and
\[
\wA\wv-\wf=S^*(AS\wv -f)
\]
we have $S\wu_s=u_s$, i.e., this proves \eqref{eq:MINRES-equiv} and \eqref{eq:MINRES-equiv-norm}.
\end{proof}

\subsection{GMRES}

The standard GMRES method minimizes at
each step the euclidean norm of the residual.
It follows that the iterate are basis dependent and differ between the
original and the extended problem.
At the same time, the behavior of GMRES mainly depends on the spectrum of the
preconditioned problem, which is the same for the two problems,
original and extended, up to the possible kernel of the latter.

\begin{theorem}\label{thm:preconditioner-spectrum}
Under the assumption that  \(T\) is an isomorphism, the spectrum
$\spectrum(\wT) $  of \(\wT\), and the spectrum $\spectrum(T) $   of \(T\) are related by
	\[\spectrum(\wT)\setminus\{0\}=\spectrum(T).\]
\end{theorem}

\begin{proof}
For \(\lambda\in\spectrum(\wT)\setminus\{0\}\) there exists
\(\wv_\lambda\not\in \ker S\) such that  \(\wT\wv_\lambda=\lambda\wv_\lambda\). 
	From \eqref{eq:TS=SwT} \[T S\wv_\lambda= S\wT\wv_\lambda =\lambda S \wv_\lambda\]
	and we see that \(S\wv_\lambda\) is an eigenvector for \(T\) and \(\lambda\in\spectrum(T)\). Thus \(\spectrum(\wT)\setminus\{0\} \subseteq \spectrum(T)\).

        Similarly if \(\lambda\in\spectrum(T)\) there exists
        \(w_\lambda\ne 0\) such that  \(\lambda w_\lambda= T
        w_\lambda\). Let
          \begin{equation}
            \label{eq:wwlambda1}
            \ww_\lambda = \lambda^{-1} \wB S^* A
          w_\lambda,
          \end{equation}
          then we have
          \begin{equation}
              \label{eq:wwlambda2}
            S \ww_\lambda = \lambda^{-1} S \wB S^* A w_\lambda =
            \lambda^{-1} T  w_\lambda =  w_\lambda
          \end{equation}
          and, using \eqref{eq:wwlambda2} and then
          \eqref{eq:wwlambda1},  $  \wT \ww_\lambda =   \wB S^* A S \ww_\lambda = \wB S^* A
            w_\lambda= \lambda \ww_\lambda  $, and 
therefore  \(\lambda \in\spectrum(\wT)\).
\end{proof}

\section{Subspace correction methods}

We briefly review, in this section, the so-called \emph{successive subspace correction} (SSC) method, see e.g., \cite{ssc-apm},  and frame them into
our setting. We assume 
local preconditioners \(B_i\), for the restriction of \(A\)
to the subspaces \(V_i\),  are given,  denote their extensions to
the whole $\wV$ as $\wB_i$, that is
\begin{equation}[wBi]
\wB_i:=\begin{pmatrix}
0 & \dots & 0\\ 
 \vdots &B_i & \vdots\\
  0 & \dots &0 
  \end{pmatrix}
\end{equation}
and define the local preconditioned operators $  T_i:=  S \wB_i
S^*A,\text{ and } \wT_i:=\wB_i\wA.$  Consider then the inexact block-Gauss-Seidel
preconditioner $\wBs$, whose application \(\ww=\wBs \wv\) is given by
the following algorithm:
\begin{center}
  \begin{algorithmic}[c]
	\State $\ww\gets 0$
	\For {$j=1,\dots,r$}
	\State $\wt\gets \wB_{i_j} \wv$
 	\State $\ww\gets \ww + \wt$
	\State $\wv\gets \wv- \wA \wt$
\EndFor
\end{algorithmic}
\end{center}
where \((i_1,\dots,i_r)\) defines the sequence of subspace
corrections.  In closed form they can be written as
\begin{equation}[wB-SSC]
\wBs:=\sum_{j=1}^r \wB_{i_j} \prod_{\ell=1}^{j-1}(\wId-\wA \wB_{i_\ell}).
\end{equation}
The preconditoned operator now reads then
\begin{displaymath}
  \wTs: = \wBs  \wA  = \sum_{j=1}^r \wT_{i_j}
  \prod_{\ell=1}^{j-1}(\wId-\wT_{i_\ell}) .
\end{displaymath}  
In this case $\Bs= S \wBs S^*$ is a multiplicative Schwarz
preconditioner (or inexact solver) for $A$.

Examples of $\wBs$ are the   \emph{forward} SSC $\wBs^{\text{forw}}$,
for  \(r=n\) and \(i_j=j\), and  the \emph{backward} SSC $\wBs^{\text{back}}$
to \(r=n\) and \(i_j=n+1-j\).

A symmetric  preconditioner can also be constructed by  a suitable
composition of  $\wBs^{\text{forw}}$, $\wBs^{\text{back}}$ and the diagonal blocks of $\wA$.

\section{Numerical examples}

We consider  the solution of the (variational) Poisson problem:
\begin{equation*}
 \text{find} \; u \in V \subset  H^1_0(\Omega) \quad \text{s.t.} \int_{\Omega} \nabla u \cdot \nabla v \, d\Omega = \int_{\Omega} f v  \, d\Omega \quad \forall v \in V \, ,
\end{equation*}
where  $\Omega$ is the image of the parametric domain $$\widehat\Omega
:= [1,2] \times [\frac{\pi}{4},\frac{3 \pi}{4}] $$ under the polar
mapping $$F(\rho,\theta) = \bigl( \rho \cos(\theta), \rho \sin(\theta)
\bigr)^T,$$ and $f$ is a piecewise-constant function (with random
constant coefficients over a uniform $4^2$ tessellation of the
parametric domain).  
The test and trial space $V$ is defined as $V = V_0 + \sum_{k=1}^n \bigl( V_k^{x} + V_k^{y} \bigr)$, where:
\begin{itemize}
  \item $V_0$ is the push forward through $F$ of the spline space of degree $p$ and continuity $p-1$, built over $\widehat\Omega$ with uniform knots vectors of size $\frac{1}{2^{L+2}} (1, \frac{\pi}{2})$ (with $L \ge0$) and satisfying homogeneous Dirichlet boundary conditions. The support of this space is depicted in Fig.\ref{fig:2d_V0};
  \item $V_k^{x}$ is the push forward through $F$ of the spline space of degree $p$ and continuity $p-1$, built over $\widehat{\Omega}_k^x = [ 1,\frac{3}{2} ] \times [ \frac{\pi}{4},\frac{\pi}{4} \bigl( 1 + \frac{1}{2^{k-1}} \bigr) ]$, with uniform knots vectors of size $\frac{1}{2^{L+2}} (1, \frac{\pi}{2^{k+1}})$ (with $L \ge0$) and satisfying homogeneous Dirichlet boundary conditions.The supports of  the spaces $V_1^{x}$ and $V_2^{x}$  (with $L=0$) are depicted in Fig.\ref{fig:2d_V1x} and Fig.\ref{fig:2d_V2x}, respectively;
  \item $V_k^{y}$ is the push forward through $F$ of the spline space of degree $p$ and continuity $p-1$, built over $\widehat{\Omega}_k^y = [ 1,1 + \frac{1}{2^k} ] \times [ \frac{\pi}{4},\frac{5 \pi}{8} ] $, with uniform knots vectors of size $\frac{1}{2^{L+2}} (\frac{1}{2^k}, \frac{\pi}{2})$ (with $L \ge0$) and satisfying homogeneous Dirichlet boundary conditions. The support of the spaces $V_1^{y}$ and $V_2^{y}$ (with $L=0$) are depicted in Fig.\ref{fig:2d_V1y} and Fig.\ref{fig:2d_V2y}, respectively.
\end{itemize}

\begin{figure}
  \centering
  \includegraphics[width=.4\linewidth]{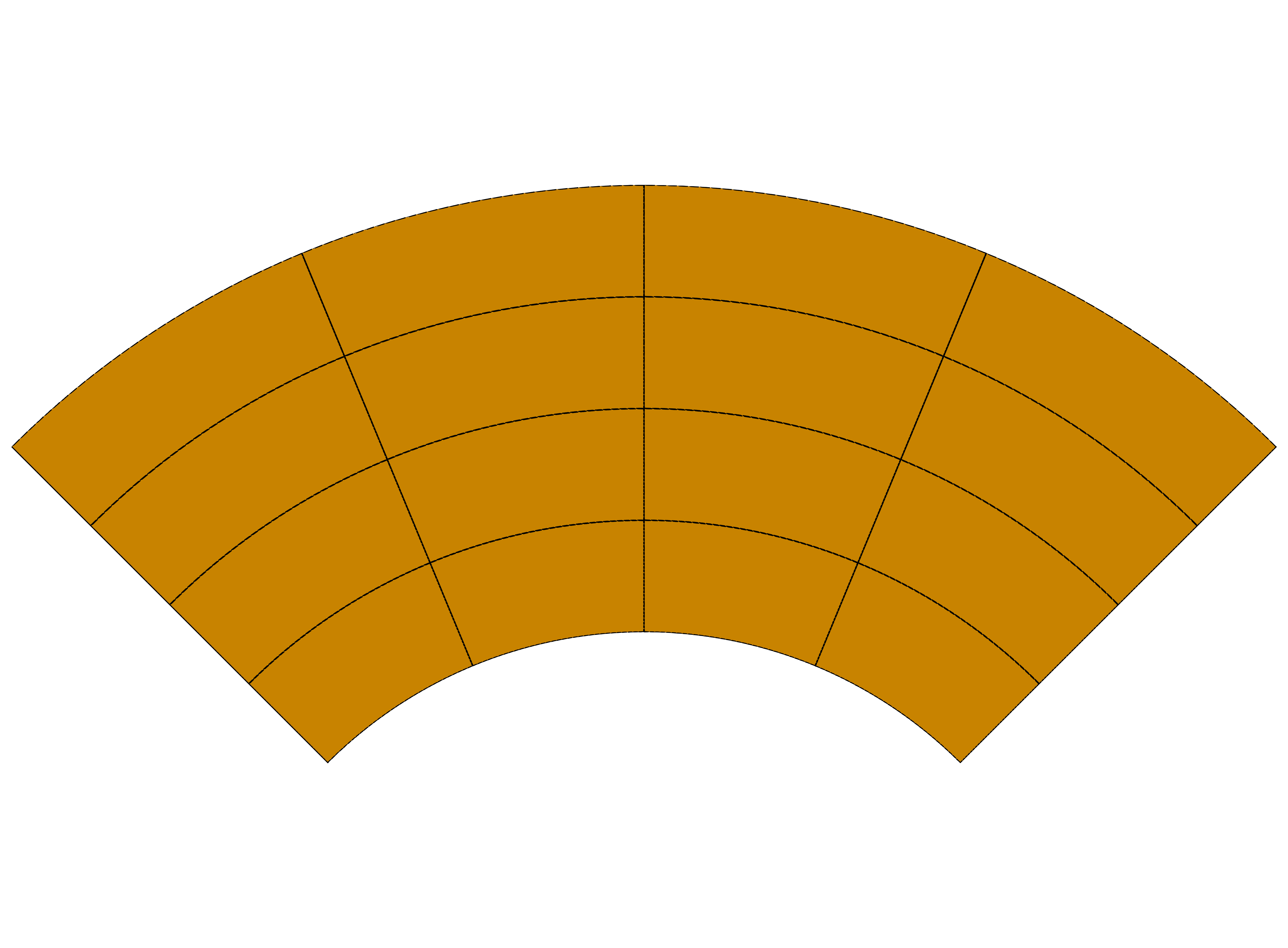}
  \caption{2D case: the physical domain $F(\widehat{\Omega}_0)$ associated to the space $V_0$ for $L=0$.}
  \label{fig:2d_V0}
\end{figure}

The iterative linear solver is a symmetric Gauss-Seidel preconditioned
Conjugate Gradient, stopped when the discrete residual is $< 10^{-11}$
or the number of iterations reached the upper limit of $1000$.
We tested the proposed method with $n = 2$ (corresponding to 5
overlapping domains) with different values of the mesh parameter $L$
for degrees $p=2$ (Table~\ref{table:aniso_2D_deg2_n2}) and $p=3$
(Table~\ref{table:aniso_2D_deg3_n2}), where we observe the
effectiveness of the preconditioned solver: in both cases $p=2$ and
$p=3$ the number of iterations needed for the convergence of the
preconditioned approach increases  for the coarser meshes (i.e. low values of $L$) and then remain almost constant for the finer meshes, as expected. 

\begin{figure}
\begin{subfigure}{.5\textwidth}
  \centering
  \includegraphics[width=.8\linewidth]{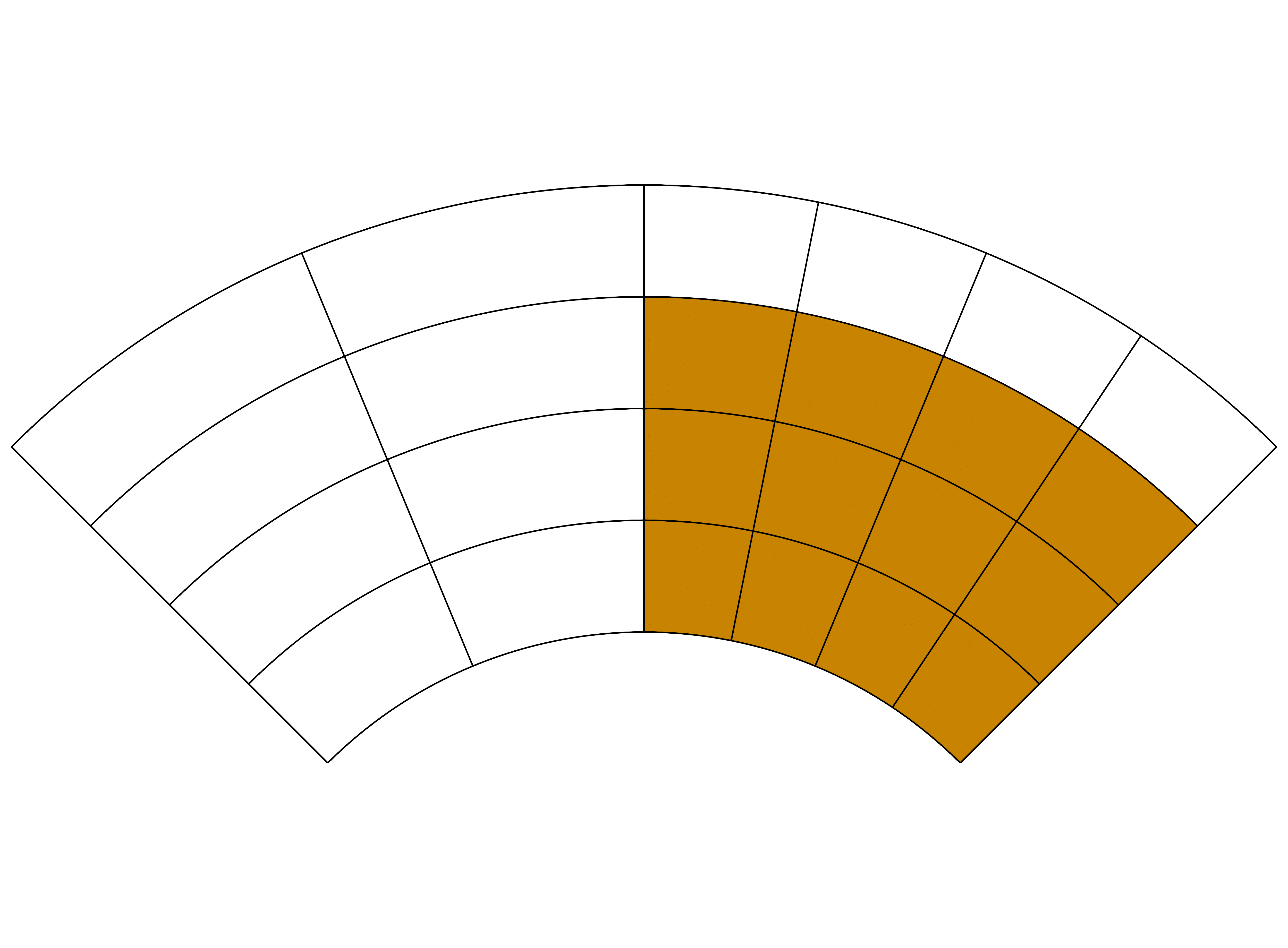}
  \caption{$F(\widehat{\Omega}_{1,x})$}
  \label{fig:2d_V1x}
\end{subfigure}
\begin{subfigure}{.5\textwidth}
  \centering
  \includegraphics[width=.8\linewidth]{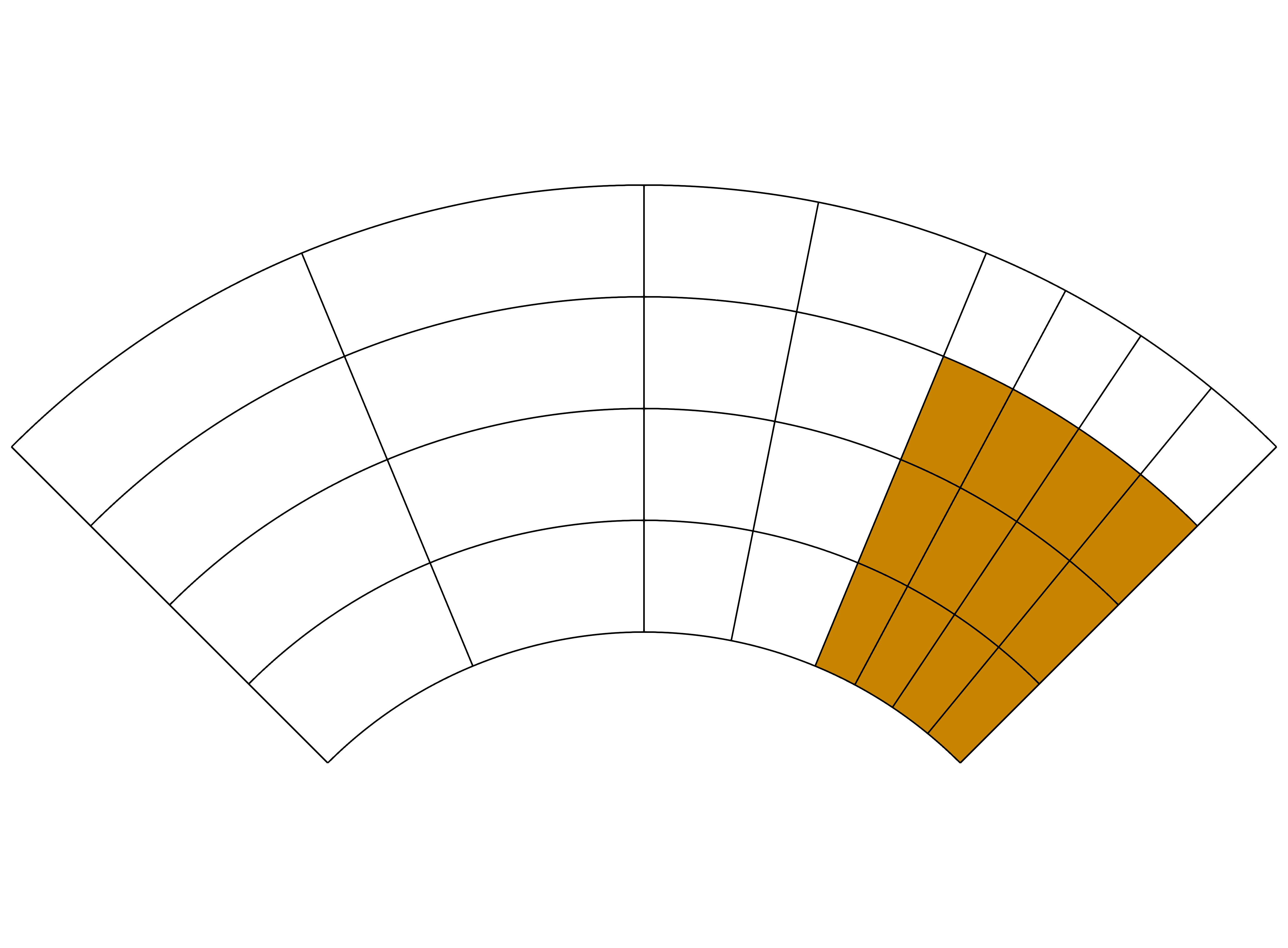}
  \caption{$F(\widehat{\Omega}_{2,x})$}
  \label{fig:2d_V2x}
\end{subfigure}
\vfill
\begin{subfigure}{.5\textwidth}
  \centering
  \includegraphics[width=.8\linewidth]{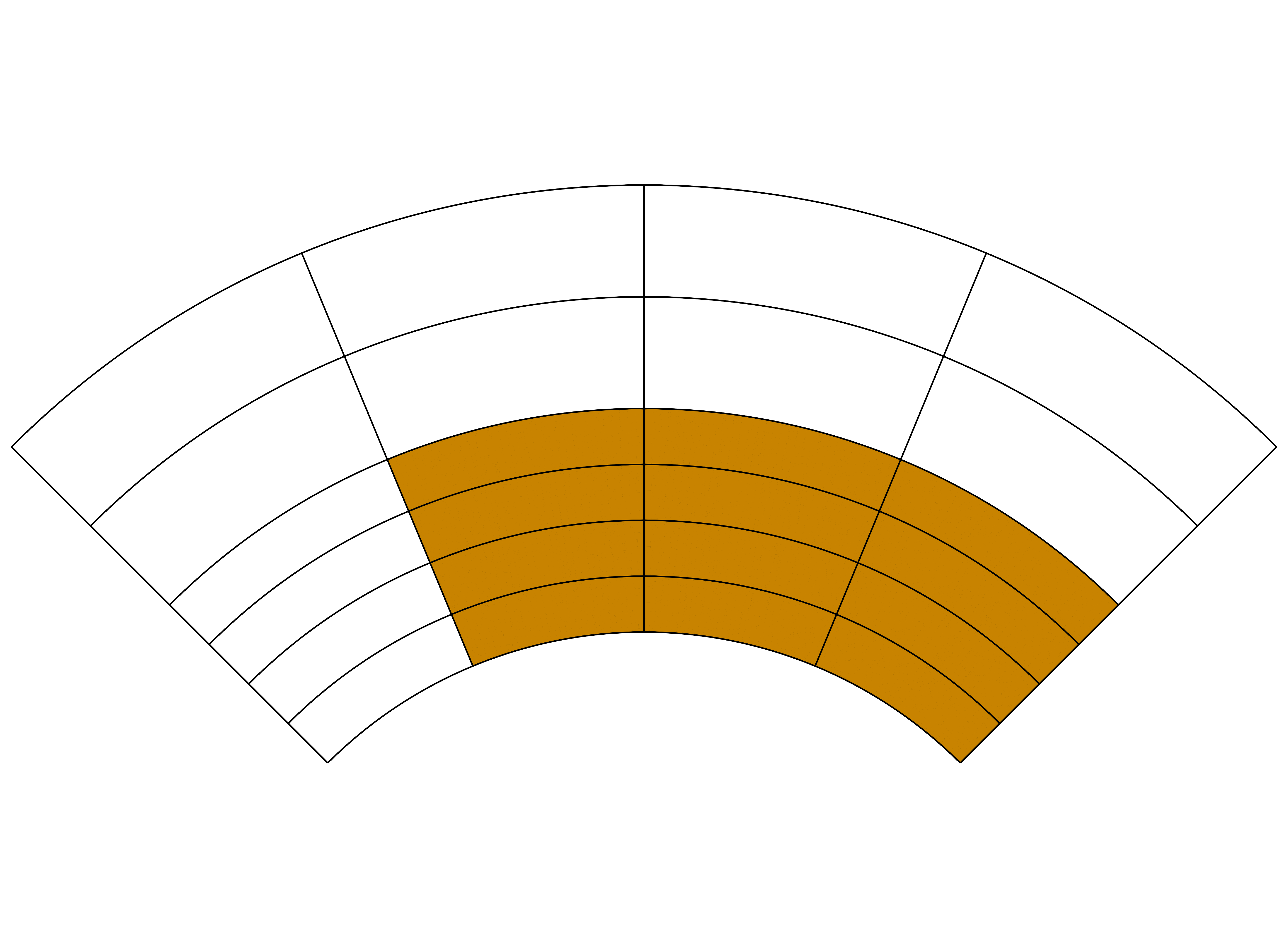}
  \caption{$F(\widehat{\Omega}_{1,y})$}
  \label{fig:2d_V1y}
\end{subfigure}%
\begin{subfigure}{.5\textwidth}
  \centering
  \includegraphics[width=.8\linewidth]{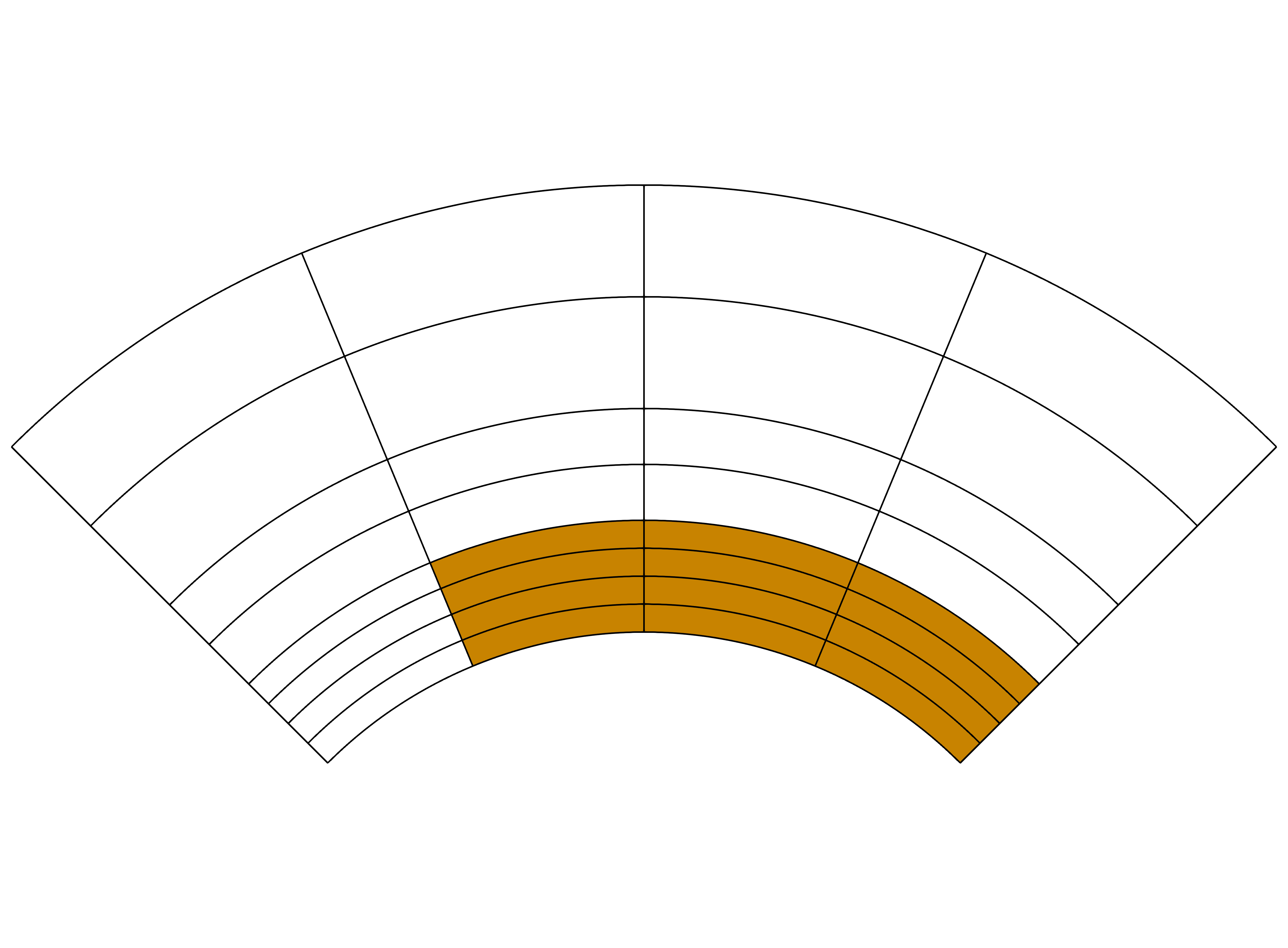}
  \caption{$F(\widehat{\Omega}_{2,y})$}
  \label{fig:2d_V2y}
\end{subfigure}%
\caption{2D case: the different physical domains associated to the spaces $V_{1}^x$ (\ref{fig:2d_V1x}), $V_{2}^x$ (\ref{fig:2d_V2x}), $V_{1}^y$ (\ref{fig:2d_V1y}) and $V_{2}^y$ (\ref{fig:2d_V2y}) for $L=0$.}
\label{fig:2d_phys_domains}
\end{figure}

\pgfplotstableset{create on use/NewColL2D/.style={create col/set list={0,1,2,3,4,5,6,7},},columns/NewColL2D/.style={string type},}

\pgfplotstableread[skip first n=1]{./numerical_tests/convergence_results_2D_anisotropic_2_deg_min=2_deg_max=2_nlevels_min=2_nlevels_max=2_polar_map_gauss_seidel.dat}{\dataAnisoTwoDdegTwo}

\begin{table}
  \centering
  \pgfplotstabletypeset[
    skip first n=1,
    columns={NewColL2D,7,9,10,11,12},
    every head row/.style={
      before row={
      \toprule
      & & \multicolumn{2}{c}{CG} & \multicolumn{2}{c}{PCG-GS} \\
      },
      after row=\midrule,
    },
    every last row/.style={
      after row=\bottomrule},
    sci zerofill,
    precision=2,
    columns/NewColL2D/.style={
      column name=$L$,
      int detect},
    columns/7/.style={
      column name=DOFs,
      int detect},
    columns/9/.style={
      column name=N. iters,
      int detect
      }, 
    columns/10/.style={
      column name=Residual}, 
    columns/11/.style={
      column name=N. iters,
      int detect}, 
    columns/12/.style={
      column name=Residual}, 
  ]
  {\dataAnisoTwoDdegTwo}
  \caption{Number of iterations
  w.r.t. the mesh parameter $L$, for the case with degree $p = 2$ and $n=2$ (i.e. $5$ overlapping domains).} \label{table:aniso_2D_deg2_n2}
\end{table}

\pgfplotstableread[skip first n=1]{./numerical_tests/convergence_results_2D_anisotropic_2_deg_min=3_deg_max=3_nlevels_min=2_nlevels_max=2_polar_map_gauss_seidel.dat}{\dataAnisoTwoDdegThree}

\begin{table}
  \centering
  \pgfplotstabletypeset[
    skip first n=1,
    columns={NewColL2D,7,9,10,11,12},
    every head row/.style={
      before row={
      \toprule
      & & \multicolumn{2}{c}{CG} & \multicolumn{2}{c}{PCG-GS} \\
      },
      after row=\midrule,
    },
    every last row/.style={
      after row=\bottomrule},
    sci zerofill,
    precision=2,
    columns/NewColL2D/.style={
      column name=$L$,
      int detect},
    columns/7/.style={
      column name=DOFs,
      int detect},
    columns/9/.style={
      column name=N. iters,
      int detect
      }, 
    columns/10/.style={
      column name=Residual}, 
    columns/11/.style={
      column name=N. iters,
      int detect}, 
    columns/12/.style={
      column name=Residual}, 
  ]
  {\dataAnisoTwoDdegThree}
  \caption{Number of iterations
  w.r.t. the mesh parameter $L$, for the case with degree $p = 3$ and $n=2$ (i.e. $5$ overlapping domains).} \label{table:aniso_2D_deg3_n2}
\end{table}

\section{Conclusions}
In this work, we proposed a framework based on well-known tools such
as Krylov solvers and preconditioners based on subspace
corrections. The novelty we introduce is working with a linear system
obtained from the disjoint union of the degrees of freedom of the
subspaces. This approach has the advantage of preserving the original
structure of the subspaces, for example, the tensor structure in the
case of isogeometric discretizations, possibly allowing efficient
(exact or inexact) solvers on the subspaces. We believe this is a promising direction, although in this work it is only sketched out and deserves further investigation and development.

\section*{Acknowledgments}

The   authors are members of GNCS-INDAM.
MM and GS acknowledge support from
PNRR-M4C2-I1.4-NC-HPC-Spoke6.  AB   acknowledges the contribution of
the Italian  MUR through the PRIN project COSMIC
(No. 2022A79M75). GS acknowledges support from the PRIN 2022 PNRR project NOTES
(No. P2022NC97R).

\bibliographystyle{plain}
\bibliography{refs}

\end{document}